\documentclass[reqno,a4paper,11pt]{article}
\usepackage{amsmath,amssymb,amsthm,amsfonts,mathrsfs}

\theoremstyle{plain}
\newtheorem{theorem}{Theorem} [section]

\newtheorem{lemma}[theorem]{Lemma}
\newtheorem{proposition}[theorem]{Proposition}

\theoremstyle{definition}
\newtheorem{definition}[theorem]{Definition}

\theoremstyle{remark}

\newtheorem{remark}[theorem]{Remark}
\numberwithin{theorem}{section}
\numberwithin{equation}{section}
\numberwithin{figure}{section}

\def\N{\mathbb N}
\def\Z{\mathbb Z}

\def\R{\mathbb R}

\def\a{\alpha}
\def\b{\beta}

\def\d{\delta}
\def\e{\varepsilon}

\def\var{\varphi}
\def\Om{\Omega}

\renewcommand{\H}{\mathcal H}
\newcommand{\K}{\mathcal K}

\DeclareMathOperator{\dist}{dist}
\DeclareMathOperator{\co}{co}

\title{Quantitative stability for sumsets in $\R^n$}

\author{Alessio Figalli\thanks{The University of Texas at Austin,
Mathematics Dept. RLM 8.100,
2515 Speedway Stop C1200,
Austin, TX 78712-1202 USA.\, \textit{E-mail address:} \texttt{figalli@math.utexas.edu}} \,\,and David Jerison\thanks{Massachusetts Institute of Technology,
77 Massachusetts Ave, Room E17-308,
Cambridge, MA 02139-4307
USA. \, \textit{E-mail address:} \texttt{jerison@math.mit.edu}}}
\date{}

\begin{document}
\maketitle

\begin{abstract}
Given a measurable set  $A\subset \R^n$ of positive measure, it is not difficult to show that
$|A+A|=|2A|$ if and only if $A$ is equal to its convex hull minus a set of measure zero.
We investigate the stability of this statement: If $(|A+A|-|2A|)/|A|$ is small,
is $A$ close to its convex hull?  Our main result is an explicit control, 
in arbitrary dimension, on the measure of the difference between 
$A$ and its convex hull in terms of $(|A+A|-|2A|)/|A|$.
\end{abstract}

\section{Introduction}
Let $A$ and $B$ be measurable subsets of $\R^n$, and let $c>0$.
Define the set sum and scalar multiple by 
\begin{equation}
\label{eq:def sum}
A + B := \{a + b: a\in A, \ b\in B\}, \quad cA := \{ ca: a\in A\} 
\end{equation}
Let $|A|$ denote the Lebesgue measure of $A$, and assume that $|A|>0$.  It is clear that 
$\frac12(A+A) \supset A$, so, in particular, 
$ \left|\frac12(A+A)\right|\geq |A|$, and 
it is not difficult to show that $\left|\frac12(A+A)\right|= |A|$
implies that $A$ is equal to its convex hull minus a set of measure zero.
(Notice that, even if $A$ is measurable, $A+A$ may not be.  In this case,
$|A+A|$ denotes the outer Lebesgue measure of $A+A$.)

The goal of this paper is to investigate whether this statement is stable.
Let us define the \textit{deficit} of $A$ as
$$
\delta(A):=\frac{\left|\frac12(A+A)\right|}{|A|}-1=\frac{\left|A+A\right|}{|2A|}-1,
$$
The question we address is whether small deficit implies that $A$
is close to its convex hull.

This question has already been extensively investigated in the
one dimensional case.   If one approximates sets in $\R$ 
with finite unions of intervals, then one can
translate the problem to $\Z$, and in the discrete setting the
question becomes a well studied problem in additive combinatorics.
There are many results on this topic, usually called Freiman-type
theorems; we refer to the book \cite{bookTao} for a comprehensive list
of references. Our problem can be seen as a very particular case.

The precise statement in one dimension is the following.
\begin{theorem}
\label{thm:Freiman}
Let $A\subset \R$ be a measurable set, and denote by $\co(A)$ its convex hull.
Then
$$
|A+A|-2|A| \geq \min\{|\co(A)\setminus A|,|A|\},
$$
or, equivalently, if $|A|>0$ then
$$
\delta(A) \geq \frac{1}{2}\min\left\{\frac{|\co(A)\setminus A|}{|A|},1\right\}.
$$
\end{theorem}
This theorem can be obtained as a corollary of a result
of G. Freiman \cite{freiman} about the structure of additive subsets of $\Z$.
(See \cite{freiman2} or \cite[Theorem 5.11]{bookTao} for a statement and a proof.)
However, it turns out that to prove of Theorem \ref{thm:Freiman} 
one only needs weaker results.  For convenience of the reader,
instead of relying on deep and intricate combinatorial results, 
we will give an elementary, completely self-contained proof
of Theorem \ref{thm:Freiman}.  Our proof is based on the simple
observation that a subset of $\R$ can be discretized to a subset of
$\Z$ starting at $0$ and ending at a prime number $p$.  
This may look strange from an analytic point of view, but it
considerably simplifies the combinatorial aspects.

The main result of this paper is a quantitative stability result in arbitrary dimension,
showing that a power of $\delta(A)$ dominates the measure of the difference between
$A$ and its convex hull $\co(A)$.
\begin{theorem}
\label{thm:main}
Let $n \geq 2$.
There exist computable dimensional constants $\delta_n,c_n>0$
such that if
$A\subset \R^n$ is a measurable set of positive measure 
with $\delta(A)\leq \delta_n$, then
$$
\delta(A)^{\a_n} \geq c_n\frac{|\co(A)\setminus A|}{|A|},\qquad \a_n:=\frac{1}{8\cdot 16^{n-2}n!(n-1)!}.
$$
\end{theorem}

Concerning this higher dimensional case, M. Christ 
\cite{christ1,christ2} proved that if
$|A+B|^{1/n}-|A|^{1/n}-|B|^{1/n} \to 0$, then $A$ and $B$ are both
close to some dilation of the same convex set. In particular, as a corollary 
one obtains that if
 $\delta(A) \to 0$ then
$|\co(A)\setminus A|\to 0$.
Although Christ's result does not imply any quantitative estimate for our problem, in another
direction it is more general, since it represents a
qualitative stability for the Brunn-Minkowski inequality.  Furthermore,
if we restrict $A$ and $B$ to the class of convex sets, a quantitative stability 
estimate for the Brunn-Minkowski inequality,
with the sharp power law  dependence on the deficit, was proved 
in \cite{fmpK,fmpBM}.   

Much of the difficulty in Christ's work arises
from the fact that he is dealing with different sets $A$ and $B$.  In
Appendix B, we show how his methods yield a relatively quick proof of
qualitative stability when $A=B$ is bounded.
Our purpose here is to provide a \emph{quantitative} stability estimate,
and since the argument is very involved, we decided to focus on the case
$A=B$.

Although in a few places our arguments may resemble
those of Christ, our strategy and most elements of our proof 
are very different, and his techniques and ours can be seen as complementary.
Indeed, as shown in a sequel to this paper \cite{fjBM}, 
a combination of them with the results from \cite{fmpK,fmpBM}
(and several new ideas)
makes it possible to prove that if $|A+B|^{1/n}-|A|^{1/n}-|B|^{1/n}$
is small relatively to the measure of $A$ and $B$, then both $A$
and $B$ are close, in a quantitative way, to dilations of the same convex set, yielding 
a
proof of quantitative stability of the Brunn-Minkowski inequality 
for measurable sets.\\

The paper is structured as follows.  In the next section we prove
Theorem \ref{thm:Freiman}.  Then in Section \ref{sect:main} we prove
Theorem \ref{thm:main} by induction on the dimension.  The strategy
is outlined at the beginning of Section \ref{sect:main}.  Some of the
technical results used in the proof of Theorem \ref{thm:main} are
collected in Appendix A.  \\

\textit{Acknowledgements:}
AF was partially supported by NSF Grant DMS-0969962 and DMS-1262411.
DJ was partially supported by NSF Grant DMS-1069225 and the Stefan Bergman Trust.
AF thanks Terence Tao for pointing out to him the reference to Freiman's Theorem.
This work was begun during AF's visit at MIT during the Fall 2012.
AF wishes to thank the Mathematics Department at MIT
for its warm hospitality.

\section{The 1d case: Proof of Theorem \ref{thm:Freiman}}

To prove Theorem \ref{thm:Freiman}, we first collect some preliminary results.

\begin{definition}
Let $A,B$ be nonempty finite subset of an additive group $Z$, and let $e \in A-B$. Then we define the $e$-\emph{transform} of $A$ and $B$ as
$$
A_{(e)}:=A\cup(B+e),\qquad B_{(e)}:=B\cap (A-e).
$$
\end{definition}

The following are three simple properties of the $e$-transform (in this discrete context, $|A|$ denotes the cardinality of $A$):
\begin{equation}
\label{eq:etransf_1}
A_{(e)}+B_{(e)} \subset A+B,
\end{equation}
\begin{equation}
\label{eq:etransf_2}
|A_{(e)}|+|B_{(e)}| = |A|+|B|,
\end{equation}
\begin{equation}
\label{eq:etransf_3}
|A_{(e)}| \geq |A|,\qquad |B_{(e)}| \leq |B|.
\end{equation}

The following is a classical result about sum of sets in $\Z_p$ \cite{cauchy,davenport},
but for completeness we include its simple proof.
A key fact used in the proof is that $\Z_p$ has no nontrivial subgroups.

\begin{lemma}[Cauchy-Davenport inequality]
\label{eq:lemmaCD}
If $Z=\Z_p$ with $p$ prime, then
$$
|A+B|\geq \min\{|A|+|B|-1,p\}.
$$
\end{lemma}
\begin{proof}
 The proof is by induction on the size of $|B|$, the case $|B|=1$ being trivial.
 
 For the induction step, we consider two cases:\\
 \textit{Case 1: there exists $e \in A-B$ such that $|B_{e}|<|B|$.}
 Then by the inductive step
 $$
 |A_{(e)}+B_{(e)}|\geq \min\{|A_{(e)}|+|B_{(e)}|-1,p\},
 $$
 and we conclude by \eqref{eq:etransf_1} and \eqref{eq:etransf_2}.\\
 \textit{Case 2: $|B_{e}|=|B|$ for any $e \in A-B$.}
 This means that $B_{(e)}=B$ for any $e \in A-B$, which implies that $B+e\subset A$ for any $e \in A-B$, that is
 $$
 A+B-B\subset A.
 $$
Thus $B-B$ is contained inside $Sym_1(A):=\{h \in \Z_p:A+h=A\}$.
(Notice that, since $|A+h|=|A|$, the inclusion $A+h\subset A$ is equivalent to the equality
$A+h=A$.)
 
It is a general fact (easy to check) that $Sym_1(A)$ is a subgroup.
Since in our case the only subgroups are $\{0\}$ and $\Z_p$, and we are assuming that $|B|>1$, this means that $Sym_1(A)=\Z_p$, so $A=\Z_p$ and the result
is true since $|A+B|\geq |A|=p$.
\end{proof}

We can now prove the following important result, which is just
a special case of  Freiman's ``$3k-3$ Theorem'' \cite{freiman,freiman2}.

\begin{proposition}
\label{prop:3k3}
Let $A$ be a finite nonempty subset of $\Z$ with $\min(A)=0$ and $\max(A)=p$, with $p$ prime. Assume that $|A+A|<3|A|-3$.
Then $|\{0,\ldots,p\}\setminus A|\leq |A+A|-2|A|+1$.
\end{proposition}
\begin{proof}
Since the cases $|A|=1,2$ are trivial, we can assume $|A|\geq 3$.
We want to show that $p\leq |A+A|-|A|$. 

Let $\phi_p:\Z\to\Z_p$ denote the canonical quotient map.
We claim that
\begin{equation}
\label{eq:phip AA}
|\phi_p(A+A)| \leq |A+A|-|A|.
\end{equation}
Indeed, $A+A$ can be written as the disjoint union of the three sets
$$
(A+A)\cap [0,p-1],\qquad A+p,\qquad \bigl((A+A)\cap [p,2p-1]\bigr)\setminus (A+p),
$$
hence
\begin{equation}
\label{eq:phip AA1}
|A+A|-|A| \geq \bigl|(A+A)\cap [0,p-1]\bigr|+\bigl|\bigl((A+A)\cap [p,2p-1]\bigr)\setminus (A+p)\bigr|.
\end{equation}
In addition, since $\phi_p(A+p)=\phi_p(A)\subset \phi_p\bigl((A+A)\cap [0,p-1]\bigr)$ (because $0 \in A$), we have
$$
\phi_p(A+A)=\phi_p\bigl((A+A)\cap [0,p-1] \bigr) \cup \phi_p\bigl(\bigl((A+A)\cap [p,2p-1]\bigr)\setminus (A+p)\bigr),
$$
which implies that
\begin{equation}
\label{eq:phip AA2}
|\phi_p(A+A)| \leq \bigl|(A+A)\cap [0,p-1]\bigr|+\bigl|\bigl((A+A)\cap [p,2p-1]\bigr)\setminus (A+p)\bigr|.
\end{equation}
Combining \eqref{eq:phip AA1} and \eqref{eq:phip AA2}, we obtain \eqref{eq:phip AA}.

By \eqref{eq:phip AA} and the hypothesis $|A+A|<3|A|-3$, we get (observe that $|\phi_p(A)|=|A|-1$)
$$
|\phi_p(A+A)|<2|A|-3=2|\phi_p(A)|-1,
$$
so by the Cauchy-Davenport inequality (Lemma \ref{eq:lemmaCD}) we deduce that
$$
|\phi_p(A+A)|=p.
$$
Using \eqref{eq:phip AA} again, this gives
$$
|A+A|-|A| \geq p,
$$
concluding the proof.
\end{proof}

We can now prove Theorem \ref{thm:Freiman}.

\begin{proof}[Proof of Theorem \ref{thm:Freiman}] We may assume that
$|A+A|-2|A|<|A|$, otherwise there is nothing to prove.  After
dilation and translation, we can also assume $(0,1)\subset \co(A) \subset [0,1]$.
We prove the result in two steps.

\textit{Step 1: $A$ is compact.}
Since $A$ is compact, so is $\co(A)$,
and so the inclusion $(0,1)\subset \co(A)\subset [0,1]$ implies that $\co(A)=[0,1]$.

Take $\{p_k\}_{k \in \N}$ a sequence of prime numbers
tending to infinity, and for each $k$ define the family of closed intervals
$$
I_{k,j}:=\left[\frac{j}{p_k+1},\frac{j+1}{p_k+1}\right],\qquad j=0,\ldots,p_k.
$$
Consider now the sets
$$
A_k:=\bigcup_{j\,:\,I_{k,j}\cap A\neq\emptyset} I_{k,j}.
$$
Observe that $A_k\supset A$ by construction. In addition, since $A$ is compact (and so
also $A+A$ is compact),
one can easily check that 
$$
\bigcap_{k=1}^\infty \bigcup_{\ell=k}^\infty A_\ell=A,\qquad \bigcap_{k=1}^\infty
\bigcup_{\ell=k}^\infty (A_\ell+A_\ell)=(A+A),
$$
from which it follows that
$$
\left|\bigcup_{\ell=k}^\infty A_\ell \right| \to |A|,\qquad \left|\bigcup_{\ell=k}^\infty (A_\ell+A_\ell)\right|\to
|A+A|\qquad\text{as $k \to \infty$}.
$$
In particular, since $A_k\supset A$ for any $k$, this implies
\begin{equation}
\label{eq:convergence measures}
|A_k|\to |A|,\qquad |A_k+A_k|\to |A+A|\qquad\text{as $k \to \infty$}.
\end{equation}
Since $|A+A|-2|A|<|A|$ and $1/p_k\to 0$, 
it follows from \eqref{eq:convergence measures} that
\begin{equation}
\label{eq:Ak small}
|A_k+A_k|-2|A_k| < |A_k|-\frac{3}{p_k+1}
\end{equation}
for $k$ sufficiently large. 

Let us consider the sets
$$
B_k:=\{j \in \Z:I_{k,j} \subset A_k\}.
$$
Recalling that $\co(A)=[0,1]$, it is easy to check that $\min(B_k)=0$ and $\max(B_k)=p_k$.
In addition, it follows immediately from \eqref{eq:Ak small} that $|B_k+B_k|<3|B_k|-3$.
Hence we can apply Proposition \ref{prop:3k3} to deduce that
$$
|\{0,\ldots,p_k\}\setminus B_k|\leq |B_k+B_k|-2|B_k|+1,
$$ 
which expressed in terms of $A_k$ becomes
$$
|[0,1]\setminus A_k|\leq |A_k+A_k|-2|A_k|+\frac{1}{p_k+1}.
$$
Letting $k \to \infty$ and using \eqref{eq:convergence measures} proves the result
when $A$ is compact.

\textit{Step 2: $A$ is a measurable set.}
This case will follow easily by inner approximation.
Let $\{A_k\}_{k \in \N}$ be an increasing sequence of compact sets contained in $A$ 
such that $|A_k| \to |A|$.  

Also, let $\{a_k,b_k\}_{k \in \N} \subset A$ be sequences of points
such that $[a_k,b_k] \to \co(A)$ as $k\to \infty$.
Then, up to replacing $A_k$ with $A_k\cup\left(\cup_{1\leq j \leq k}\{a_j,b_j\}\right)$, we can assume that  $\co(A_k)\supset [a_k,b_k]$,
so that in particular
$$
|\co(A)\setminus \co(A_k)| \to 0\qquad \text{as $k\to \infty$}.
$$
In addition, since $A_k\subset A$, $|A_k+A_k|\leq |A+A|$.  Therefore,
$|A+A|-2|A|<|A|$ implies that $|A_k+A_k|-2|A_k|<|A_k|$ for $k$ 
sufficiently large.  So, by Step 1,
$$
|\co(A_k)\setminus A_k| \leq |A_k+A_k|-2|A_k| \leq |A+A|-2|A_k|,
$$
and letting $k \to \infty$ concludes the proof.
\end{proof}

\section{The induction step: Proof of Theorem \ref{thm:main}}
\label{sect:main}

Let $\H^k$ denote the $k$-dimensional Hausdorff measure on $\R^n$.  
Denote by $(y,t) \in \R^{n-1}\times \R$ a point in $\R^n$,
and denote by $\pi:\R^n \to \R^{n-1}$ the canonical projection $\pi(y,t):= y$.
Given $E\subset \R^n$ and $y \in \R^{n-1}$, we use the notation 
\begin{equation}
\label{eq:Ey}
E_y := E\cap \pi^{-1}(y).
\end{equation}
We say that $E$ is {\em $t$-convex}  if $E_y$ is a segment for
every $y \in \pi(E)$.

Throughout the proof, $C$ will denote a generic constant depending only 
on the dimension, which may change from line to line.
The proof of the inductive step is long and involved, so we will
divide it into several steps and sub-steps.  The strategy is outlined here.

Notice that, as in the proof of Theorem \ref{thm:Freiman}, it suffices to 
consider only the case when $A$ is compact, since the general case follows easily by inner approximation.

In Step 1 we replace $A$ by a $t$-convex 
set $A^*$ as follows.  By applying Theorem  \ref{thm:Freiman}
to the sets $\{A_y\}_{y \in \pi(A)}$, we deduce that most of 
these subsets of $\R$ are close to their convex hull.
Also, by Theorem \ref{thm:main} applied with $n-1$, we can find a small 
constant $s_0>0$ such that the set $E_{s_0}:= \{y:\H^{1}(A_y)>s_0\}$ is close to a 
convex set. This allows us to construct a $t$-convex set $A^*$ 
whose sections consist
of ``vertical'' segments $\co(A_y)=\{y\}\times [a(y),b(y)] $ 
at most points $y$ of $E_{s_0}$.  We then show that $A^*$ is
close to $A$ and has several other nice geometric properties.  These
properties lead, in Step 2, to the fact that the midpoints $c(y):=(a(y)+b(y))/2$ of
the sections of  $A^*$ have bounded second differences 
as a function of $y$.

In Step 3, we show that, after an affine transformation of
determinant $1$, $A^*$ can be assumed to be bounded.  Observe that such
transformations preserve the Lebesgue measure $|A|$, the deficit
$\delta(A)$, and, of course, the property of convexity.  To carry out
the third step, the geometric properties of $A^*$ once again play a
crucial role.  Near convexity of $E_{s_0}$ immediately shows that there is
a linear transformation in $y$ that makes $E_{s_0}$ bounded.  Showing that
one can subtract a linear function from $c(y)$ so that it is bounded
is more complicated.  We prove first that that this is the case 
on a significant fraction of $E_{s_0}$ using the bound on second differences
of $c$.  Then, using this bound again allows us to control $c$ at every point.
A similar estimate already appeared in the works \cite{christ1, christ2},
although here we need to use a different strategy to obtain quantitative 
bounds.

In Step 4 we show that $A^*$ is close 
to a convex set.  The proof relies not only on 
the geometric properties of $A^*$ established in the preceding parts, but
also on a further application of Theorem \ref{thm:main} with $n-1$ to 
the level sets of the functions $a(y)$ and $b(y)$.

Finally, in the last step we show that such a convex set can be 
assumed to be the convex hull of $A$. This will conclude the proof.\\

\subsection*{Step 1: $A$ is close to a $t$-convex set}
We assume that we already proved Theorem \ref{thm:main} through
$n-1$, and we want to show its validity for $n$.
So, let $A\subset \R^n$ be a compact set (recall that, by inner approximation,
it is sufficient to consider this case), and assume 
without loss of generality that $|A|=1$.
We show in this step that, if $\delta(A)$ is sufficiently small 
(the smallness depending only on the dimension),
there exists a $t$-convex set $A^*\subset \R^n$ such that
$$
|A^* \Delta A| \leq C\,\delta(A)^{1/2}.
$$
(Here and in the sequel, $E\Delta F$ denotes the symmetric difference between $E$ and $F$,
that is $E\Delta F=(E\setminus F)\cup(F\setminus E)$.)

\subsubsection*{Step 1-a: Most of the sections $A_y$ are close to their convex hull}
Since $\d(A)$ is finite, it follows from  Lemma \ref{lemma:observation} 
(applied with $k=n-1$,  $\pi = \pi_k$)
that $\sup_{y \in \R^{n-1}}\H^1(A_y)<\infty$ (see \eqref{eq:Ey}).
Hence, up to a linear transformation of the form
$$
(y,t)\mapsto (\lambda y,\lambda^{1-n}t),\qquad \lambda >0,
$$
we can assume that 
\begin{equation}
\label{eq:Ay}
\sup_{y \in \R^{n-1}}\H^1(A_y)=1.
\end{equation}
(Observe that, since the transformation has determinant one, both $|A|$ and $\delta(A)$ are unchanged.)
With this renormalization, using Lemma \ref{lemma:observation}
again, we deduce that 
\begin{equation}
\label{eq:meas piA}
\H^{n-1}\bigl(\pi(A)\bigr) \leq 2^{n+1}
\end{equation}
provided $\delta(A)$
is sufficiently small.

Let us write $\pi(A)$ as $F_1\cup F_2$, where
\begin{equation}
\label{eq:F1F2}
F_1:=\bigl\{y \in \pi(A): \H^1(A_y+A_y)-2\H^1(A_y)< \H^1(A_y)\bigr\},
\qquad F_2:=\pi(A)\setminus F_1.
\end{equation}
Let us notice that, by definition, $A_y\subset \{y\}\times \R \subset \R^n$ (see \eqref{eq:Ey}),
so the following set inclusions hold:
\begin{equation}
\label{eq:inclusions}
2A_y \subset A_y+A_y \subset (A+A)_{2y}.
\end{equation}
(Here and in the sequel we use $2A_y$ to denote
the set $2(A_y)$, which by definition \eqref{eq:def sum}
is a subset of $\{2y\}\times \R$.)
Since by Fubini's Theorem
\begin{equation}
\label{eq:Fubini}
\int_{\R^{n-1}}\H^{1}\bigl((A+A)_{2y}\setminus 2A_y\bigr)\,dy=|A+A|-|2A| = 2^n\delta(A),
\end{equation}
by \eqref{eq:F1F2}, \eqref{eq:inclusions}, and
Theorem \ref{thm:Freiman} applied to each set $A_y\subset\R$, we deduce that 
\begin{equation}
\label{eq:freiman Ay}
\int_{F_1}\H^{1}\bigl(\co(A_y)\setminus A_y\bigr)\,dy+ \int_{F_2}\H^1(A_y)\,dy \leq 2^n\delta(A).
\end{equation}
Let $F_1'\subset F_1$ denote the set of $y \in F_1$ such that 
\begin{equation}
\label{eq:F1'}
\H^{1}\bigl(\co(A_y)\setminus A_y\bigr) +\H^{1}\bigl((A+A)_{2y}\setminus 2A_y\bigr)\leq \d(A)^{1/2},
\end{equation}
and notice that, by \eqref{eq:Fubini}, \eqref{eq:freiman Ay},
and Chebyshev's inequality,
\begin{equation}
\label{eq:F1F1'}
\H^{n-1}(F_1\setminus F_1')\leq C \,\delta(A)^{1/2}.
\end{equation}

\subsubsection*{Step 1-b: Most of
the levels sets $\{\H^1(A_y)>s\}$ are close to their convex hull}
Here we apply the inductive step to the function
$$
y\mapsto \H^1(A_y)
$$
to deduce that most of its level sets are almost convex.
More precisely, let us define
$$
E_s:=\bigl\{y\in \R^{n-1}:\H^1(A_y)>s\bigr\}, \qquad s>0.
$$
Observe that, because of \eqref{eq:Ay}, $E_s$ is empty for $s>1$.
In addition, recalling \eqref{eq:inclusions}, it is immediate to check that
$$
E_s+E_s\subset \bigl\{ y\in \R^{n-1}:\H^1\bigl((A+A)_{2y}\bigr)>2s\bigr\},
$$
so, by Fubini's Theorem and \eqref{eq:Fubini},
\begin{equation}
\label{eq:Es sum}
2\int_0^1 \bigl(\H^{n-1}(E_s+E_s) - \H^{n-1}(2E_s)\bigr) \,ds \leq |A+A|-|2A| = 2^n\delta(A).
\end{equation}
Let $\delta_{n-1}>0$ be given by Theorem \ref{thm:main} with $n-1$ in place of $n$,
and let us partition $[0,1]$ as $G_1\cup G_2$, where
$$
G_1:=\bigl\{s \in [0,1]:\H^{n-1}(E_s+E_s)-2^{n-1}\H^{n-1}(E_s)< 2^{n-1}\delta_{n-1}\H^{n-1}(E_s)\bigr\},
\qquad G_2:=[0,1]\setminus G_1.
$$
Then, by Theorem \ref{thm:main} applied with $n-1$ we get
\begin{equation}
\label{eq:freiman Es}
c_{n-1}^{1/\alpha_{n-1}}\int_{G_1} \frac{\H^{n-1}\bigl(\co(E_s)\setminus E_s\bigr)^{1/\a_{n-1}}}{\H^{n-1}(E_s)^{1/\alpha_{n-1}-1}}\,ds+ 
\delta_{n-1}\int_{G_2}\H^{n-1}(E_s)\,ds 
\leq \delta(A).
\end{equation}
Since $\H^{n-1}(E_s)\leq 2^{n+1}$ for any $s>0$ (by \eqref{eq:meas piA}),
$\int_0^1\H^{n-1}(E_s)\,ds=|A|=1$, 
and $s \mapsto \H^{n-1}(E_s)$ is a decreasing function, we easily deduce that
\begin{equation}
\label{eq:level sets}
1/2\leq \H^{n-1}(E_s) \leq 2^{n+1}\qquad  \forall \,s \in (0,2^{-n-2}).
\end{equation}
Thus, by \eqref{eq:freiman Es}, 
$$
\H^1\bigl(G_2 \cap (0,2^{-n-2})\bigr) < C\,\d(A).
$$
Hence
$$
\H^1\bigl(G_1\cap [10\d(A)^{1/2},20\d(A)^{1/2}]\bigr) \geq 9 \d(A)^{1/2}
$$
for $\delta(A)$ small enough, and
it follows from \eqref{eq:Es sum}, \eqref{eq:freiman Es}, \eqref{eq:level sets},
and Chebyshev's inequality, that
we can find a level $s_0 \in [10\delta(A)^{1/2},20\delta(A)^{1/2}]$
such that
\begin{equation}
\label{eq:Es0}
\H^{n-1}\bigl((E_{s_0}+E_{s_0})\setminus 2E_{s_0}\bigr) \leq C\,\delta(A)^{1/2},
\end{equation}
and
$$
\H^{n-1}\bigl(\co(E_{s_0})\setminus E_{s_0}\bigr)^{1/\alpha_{n-1}} \leq C\,\delta(A)^{1/2},
$$
or equivalently
\begin{equation}
\label{eq:F_1''1}
\H^{n-1}\bigl(\co(E_{s_0})\setminus E_{s_0}\bigr) \leq C\,\delta(A)^{\alpha_{n-1}/2}.
\end{equation} 
Let us define $F_1''$ to be a compact subset of $F_1'\cap E_{s_0}$ which satisfies
\begin{equation}
\label{eq:def F1''}
\H^{n-1}\bigl((F_1'\cap E_{s_0})\setminus F_1''\bigr) \leq \delta(A)^{1/2},
\end{equation}
where $E_{s_0}$ is as in \eqref{eq:F_1''1}. Notice that  
$\H^1(A_y)\geq 10\delta(A)^{1/2}$ for $y \in E_{s_0}$, so
by \eqref{eq:freiman Ay} we get
 $$
 10\,\H^{n-1}(F_2 \cap E_{s_0})\,\delta(A)^{1/2}\leq \int_{F_2\cap E_{s_0}}\H^1(A_y)\,dy \leq 2^n\delta(A),
  $$
hence 
\begin{equation}
\label{eq:Es0F2}
 \H^{n-1}(F_2 \cap E_{s_0}) \leq C\,\delta(A)^{1/2}
 \end{equation}
 Since
 $$
 F_1''\subset E_{s_0}\qquad\text{and}\qquad E_{s_0}\setminus F_1''\subset (F_1\setminus F_1') \cup \bigl((F_1'\cap E_{s_0})\setminus F_1''\bigr)\cup(F_2\cap E_{s_0}),
$$
 combining \eqref{eq:F1F1'}, \eqref{eq:F_1''1}, \eqref{eq:def F1''}, and \eqref{eq:Es0F2},
  we obtain (observing that $\a_{n-1}\leq 1$)
\begin{equation}
\label{eq:coF_1''}
\begin{split}
\H^{n-1}\bigl(\co(F_1'') \setminus F_1''\bigr) &\leq \H^{n-1}\bigl(\co(E_{s_0}) \setminus E_{s_0}\bigr)
+\H^{n-1}(E_{s_0}\setminus F_1'')\\
& \leq C\,\delta(A)^{\alpha_{n-1}/2}.
\end{split}
\end{equation}
Moreover, \eqref{eq:F1F1'},
\eqref{eq:Es0}, \eqref{eq:def F1''}, and \eqref{eq:Es0F2}, give
\begin{equation}
\label{eq:2F_1''}
\begin{split}
\H^{n-1}\bigl((F_1''+F_1'') \setminus 2F_1''\bigr) 
&\leq \H^{n-1}\bigl((E_{s_0}+E_{s_0})\setminus 2 E_{s_0}\bigr)+ 2^{n-1}\H^{n-1}(E_{s_0}\setminus F_1'')\\
&\leq C\,\delta(A)^{1/2}.
\end{split}
\end{equation}

\subsubsection*{Step 1-c: Construction of $A^*$}
Let us define
$$
A^*:= \bigcup_{y \in F_1''} \co(A_y) \subset \R^n.
$$
Observe that, since 
$\H^1(A_y) \leq 20\, \delta(A)^{1/2}$ for 
$y\in F_1'\setminus E_{s_0}$, 
by \eqref{eq:freiman Ay}, \eqref{eq:Ay}, \eqref{eq:meas piA}, \eqref{eq:F1F1'}, and \eqref{eq:def F1''},  we get
\begin{equation}
\label{eq:AA*}
\begin{split}
|A^*\Delta A| &\leq \int_{F_1''}\H^{1}\bigl(\co(A_y)\setminus A_y\bigr)\,dy+
\int_{F_2}\H^1(A_y)\,dy\\
&\qquad
+ \int_{F_1'\setminus E_{s_0}} \H^1(A_y)\,dy
+\int_{F_1\setminus F_1'}\H^1(A_y)\,dy
+ \int_{(F_1'\cap E_{s_0})\setminus F_1''} \H^1(A_y)\,dy\\
&\leq 2^n \d(A) + 20\,\d(A)^{1/2} \H^{n-1}\bigl(\pi(A)\bigr)+
\H^{n-1}(F_1\setminus F_1')+ \H^{n-1}\bigl((F_1'\cap E_{s_0})\setminus F_1''\bigr) \\
&\leq C\,\delta(A)^{1/2}.
\end{split}
\end{equation}
Also, since $\H^1(A_y^*)=\H^1\bigl(\co(A_y)\bigr)\leq 1+\delta(A)^{1/2}$ for all $y \in F_1''$ (see \eqref{eq:Ay} and \eqref{eq:F1'}),
it follows from \eqref{eq:AA*} that 
$$
1-C\,\delta(A)^{1/2}\leq |A^*| =\int_{F_1''} \H^1(A_y^*)\,dy \leq (1+\delta(A)^{1/2}) \H^{n-1}(F_1''),
$$
which implies in particular 
\begin{equation}
\label{eq:meas H}
\H^{n-1}(F_1'') \geq 1/2.
\end{equation}

\subsection*{Step 2: The sections of $A^*$ have
controlled barycenter}
Here we show that, if we write $A_y^*$ as $\{y\}\times [a(y),b(y)]$ (recall that $A_y^*=\co(A_y)$ is a segment)
and we define $c:F_1''\to \R$ to be the barycenter of $A_y^*$, that is
\[
c(y):=\frac{a(y)+b(y)}{2} \qquad \forall\, y \in F_1'',
\]
then 
\begin{equation}
\label{eq:c}
|c(y') + c(y'') - 2c(y)|\leq 6 \qquad \forall\,y,y',y'' \in F_1'',\, \,
y=\frac{y'+y''}2. 
\end{equation} 

\subsubsection*{Step 2-a: Some geometric properties of $A^*$}

First of all, since 
$$
(A+A)_{2y}=\bigcup_{2y=y'+y''} A_{y'}+A_{y''},
$$
by \eqref{eq:F1'}
we get
\begin{equation}
\label{eq:sum Ay}
\H^1\biggl(\Bigl(\bigcup_{2y=y'+y''} A_{y'}+A_{y''}\Bigr) \setminus 2A_y\biggr) \leq  \delta(A)^{1/2}\qquad \forall\,y \in F_1''.
\end{equation}
Also,
if we define the characteristic functions
$$
\chi_y(t):=\left\{
\begin{array}{ll}
1 & \text{if }(y,t) \in A_y\\
0 & \text{otherwise},
\end{array}
\right.\qquad 
\chi_y^*(t):=\chi_{[a(y),b(y)]}(t)=\left\{
\begin{array}{ll}
1 & \text{if }(y,t) \in A_y^*\\
0 & \text{otherwise},
\end{array}
\right.\qquad 
$$
then by 
\eqref{eq:F1'} we have
the following estimate on the convolution of the functions $\chi_y$ and $\chi_y^*$:
\begin{equation}
\label{eq:close Linfty}
\begin{split}
\|\chi_{y'}^*\ast \chi_{y''}^* - \chi_{y'}\ast \chi_{y''}\|_{L^\infty(\R)}&\leq
\|\chi_{{y''}}^* -  \chi_{{y''}}\|_{L^1(\R)}
+ 
\|\chi_{{y'}}^* -  \chi_{{y'}}\|_{L^1(\R)}\\
& \leq
\H^1\bigl(\co(A_{y''})\setminus A_{y''}\bigr)+
\H^1\bigl(\co(A_{y'})\setminus A_{y'}\bigr)\\
& < 3 \delta(A)^{1/2}\qquad \forall\,y',y''\in F_1''.
\end{split}
\end{equation}
Let us define $\bar \pi:\R^n \to \R$ to be the orthogonal projection onto the last 
component, that is $\bar\pi(y,t):=t$,
and denote by $[a,b]$ the interval $\bar\pi(A^*_{y'}+A^*_{y''})$. 
Notice that, since by construction
$\H^1(A_z)\geq 10\delta(A)^{1/2}$ for any $z \in F_1''$, this interval has length greater than $20\delta(A)^{1/2}$. Also,
it is easy to check that
the function $\chi_{{y'}}^*\ast \chi_{{y''}}^*$ is supported on $[a,b]$, has slope equal to $1$ (resp. $-1$)
inside $[a,a+3\delta(A)^{1/2}]$ (resp. $[b-3\delta(A)^{1/2},b]$),
and it is greater than $3\delta(A)^{1/2}$ inside $[a+3\delta(A)^{1/2},b-3\delta(A)^{1/2}]$.
Since $\bar\pi(A_{y'}+A_{y''})$ contains the set $\{\chi_{y'}\ast\chi_{y''}>0\}$, 
by \eqref{eq:close Linfty} we deduce that
\begin{equation}
\label{eq:A*y'}
\bar\pi(A_{y'}+A_{y''}) \supset [a+3\delta(A)^{1/2}, b-3\delta(A)^{1/2}].
\end{equation}
We claim that if $2y=y'+y''$ and $y,y',y'' \in F_1''$, then
\begin{equation}
\label{eq:y'y''y}
[a(y')+a(y''),b(y')+b(y'')] \subset [2a(y) - 16 \delta(A)^{1/2},2b(y)
+ 16\delta(A)^{1/2}].
\end{equation}
Indeed, if this was false, since 
$
[a(y')+a(y''),b(y')+b(y'')] =\bar\pi(A^*_{y'}+A^*_{y''})=:[a,b]$ is an interval of 
length  at least $20\delta(A)^{1/2} \geq 16\delta(A)^{1/2}$, it follows that
$$
\H^1\bigl([a,b] \setminus  [2a(y),2b(y)]\bigr)
\geq 16 \delta(A)^{1/2}.
$$
This implies that
$$
\H^1\bigl([a+3\delta(A)^{1/2}, b-3\delta(A)^{1/2}] \setminus 
 [2a(y),2b(y)]\bigr) \geq 10\delta(A)^{1/2},
$$
so by \eqref{eq:A*y'} we get (recall that $ [2a(y),2b(y)]=\bar\pi(2A^*_y)$)
\begin{equation}
\H^1\bigl((A_{y'}+A_{y''})\setminus 2A^*_y\bigr) \geq 10 \delta(A)^{1/2}.
\end{equation}
However, since $A^*_y=\co(A_y)\supset A_y$, this contradicts \eqref{eq:sum Ay}
proving the claim \eqref{eq:y'y''y}. 

\subsubsection*{Step 2-b: Estimating the second differences of $c$}
Because of \eqref{eq:Ay} and \eqref{eq:F1'}, each set $A_y^*$ is
an interval of length at most $2$.
Hence \eqref{eq:c} follows easily from \eqref{eq:y'y''y}.

\subsection*{Step 3:  After a volume-preserving affine transformation, $A^*$ is 
universally bounded}
We show that there exist linear maps $T:\R^{n-1}\to \R^{n-1}$ and $L:\R^{n-1}\to \R$,
with $\det(T)=1$, and a point $(y_0,t_0) \in \R^n$, such that the image of 
$A^*$ under the affine transformation
\begin{equation}
\label{eq:affine}
\R^{n-1} \times \R \ni (y,t) \mapsto (Ty,t-Ly)+(y_0,t_0)
\end{equation}
is universally bounded. 
Notice that such transformation has unit determinant.

\subsubsection*{Step 3-a: After a volume-preserving affine transformation, $\pi(A^*)$ is universally bounded}
We claim that, after an affine transformation of the form
$$
\R^{n-1} \times \R \ni (y,t) \mapsto (Ty,t)+(y_0,0)
$$
with $T:\R^{n-1}\to \R^{n-1}$ linear and $\det(T)=1$, we can assume that
\begin{equation}
\label{eq:Ks ball}
B_{r} \subset \co(F_1'')\subset B_{(n-1)r},
\end{equation}
where $B_{r}\subset \R^{n-1}$ is the $(n-1)$-dimensional ball
of radius $r$ centered at the origin, and $1/C_n < r < C_n$ for a constant
$C_n$ depending only on the dimension.
Since $\pi(A^*)=F_1'' \subset \co(F_1'')$, this proves in particular the boundedness of $\pi(A^*)$.

Indeed, by John's Lemma \cite{john} applied to the convex set $\co(F_1'')$,
there exist a linear map $T:\R^{n-1}\to \R^{n-1}$
with $\det(T)=1$, and a point $y_0 \in \R^{n-1}$, such that
$$
B_{r} \subset T(\co(F_1''))+y_0\subset B_{(n-1)r}.
$$

Since 
$\H^{n-1}\bigl(T(\co(F_1''))\bigr)=\H^{n-1}\bigl(\co(F_1'')\bigr)   
\geq \H^{n-1}(F_1'')\geq 1/2$ (see \eqref{eq:meas H}), we deduce that $r$ is universally bounded
from below.  For the upper bound note that, by \eqref{eq:coF_1''} and \eqref{eq:meas piA},
 the volume of 
$\co(F_1'')$ is bounded above, 
proving the claim.

\subsubsection*{Step 3-b: Selecting some ``good'' points $y_1,\ldots,y_n$ inside $F_1''$}
We claim there exists a dimensional constant $c_n>0$ such that the following holds:
we can find $n$ points $y_1,\ldots,y_n\in B_r$, with $r>0$
as in \eqref{eq:Ks ball}, such that the following conditions are satisfied:
\begin{enumerate}
\item[(a)] All points
$$
y_1,\ldots,y_n\quad \text{and} \quad \frac{y_1+y_2}2
$$
belong to $F_1''$.
\item[(b)] Let $\Sigma_i$, $i=1,\, \dots, n$, denote the $(i-1)$-dimensional 
simplex generated by $y_1,\ldots, y_i$, 
and define $\Sigma_i':=\frac12(\Sigma_i + y_{i+1})$, $i=1,\, \dots,\,  n-1$. 
Then
\begin{enumerate}
\item[(i)]
$$
\H^{i-1}(\Sigma_i) \geq c_n,\quad \frac{\H^{i-1}(\Sigma_i\cap F_1'')}{\H^{i-1}(\Sigma_i)} \geq 1-\delta(A)^{\alpha_{n-1}/3}, 
\qquad \forall\,i=2,\ldots,n;
$$
\item[(ii)]
$$
\frac{\H^{i-1}\left(\Sigma_i'\cap F_1''\right)}{\H^{i-1}\left(\Sigma_i'\right)} \geq 1-\delta(A)^{\alpha_{n-1}/3}\qquad \forall\,i=2,\ldots,n-1.
$$
\end{enumerate}
\end{enumerate}

To see this, observe that
$\H^{n-1}(B_r\setminus F_1'') \leq C \,\delta(A)^{\alpha_{n-1}/2}$ (by \eqref{eq:coF_1''}
and \eqref{eq:Ks ball}),
which implies that, for any point $y \in B_{r}\cap F_1''$,
$$
\H^{n-1}\left(B_{r}\cap F_1'' \cap \bigl(2(B_{r}\cap F_1'')-y\bigr)\right) \geq \H^{n-1}\left(B_{r}\right)-3C \,\delta(A)^{\alpha_{n-1}/2}.
$$
Hence, for $\{y_i\}_{i=1,\ldots,n}$ to satisfy (a) above, we only need to ensure that
$$
y_i \in F_1'',\quad y_2 \in F_1'' \cap (2F_1''-y_1),
$$  
while property (b) means that each simplex $\Sigma_i$ has substantial measure (i.e., the points $\{y_i\}_{i=1,\ldots,n}$
are not too close to each other), and most of the points in $\Sigma_i$ and $\Sigma_i'$ belong to $F_1''$.
Since $\Sigma_i'=\frac12(\Sigma_i + y_{i+1})$, this is equivalent to the fact that $\Sigma_i$
intersect $F_1'' \cap (2F_1''-y_{i+1})$ in a large fraction.

Since all sets $\bigl\{F_1'' \cap (2F_1''-y_i)\bigr\}_{i=1,\ldots,n}$ cover $B_r$ up to a set of measure $C \,\delta(A)^{\alpha_{n-1}/2}
\ll \delta(A)^{\alpha_{n-1}/3}$, by a simple Fubini argument
we can choose the points $\{y_i\}_{i=1,\ldots,n}$
so that both (a) and (b) are satisfied (we leave the details to the reader).

\subsubsection*{Step 3-c: A second volume-preserving affine transformation}
Let $y_1,\ldots,y_n$ be the points constructed in Step 3-b.
We may apply an affine transformation of the form
$$
\R^{n-1} \times \R \ni (y,t) \mapsto (y,t-Ly)+(0,t_0)
$$
where 
$L:\R^{n-1}\to \R$ is linear and $t_0 \in \R$, in order to
assume that 
$$
(y_k,0) \in   A^*, \quad k= 1, \ldots, n \, .
$$
We now prove that $A^*$ is universally bounded.

\subsubsection*{Step 3-d: A nontrivial fraction of $A^*$ is bounded}
We start by iteratively applying Lemma \ref{lemma:1d}(ii): Because
$(y_i,0) \in A^*$ for all $i=1,\ldots,n$ and $A^*_y$ has length at most $2$ for any $y \in F_1''$
(see \eqref{eq:Ay} and \eqref{eq:F1'}),
we deduce that 
$$
|c(y_i)| \leq 1 \qquad \forall\,i=1,\ldots,n.
$$
(recall that $c(y_i)$ is the barycenter of $A_{y_i}^*$ ).
Also, by (a) in Step 3-b we know that $\frac12(y_1+y_2)\in F_1''$, and
by (b)-(i) most of the points in the segment
$\Sigma_2=[y_1,y_2]$ belong to $F_1''$.
So, thanks to \eqref{eq:c}, we can apply Lemma \ref{lemma:1d}(ii)
to the function $[0,1]\ni \tau \mapsto \frac{1}{3}c(\tau y_1+(1-\tau)y_2)$ and deduce
that $c$ is universally bounded on $\Sigma_2\cap F_1''$.

We now use both (b)-(i) and (b)-(ii) to iterate this construction: since $c$
is universally bounded on $\Sigma_2\cap F_1''$ and at $y_3$, for any point $z \in \Sigma_2\cap F_1''$ such that
$\frac12(z+y_3) \in \Sigma_2'\cap F_1''$ (these are  most of the points)
we can apply again Lemma \ref{lemma:1d}(ii) to the function $[0,1]\ni \tau \mapsto \frac{1}{3}c(\tau z+(1-\tau)y_3)$
to deduce that $c$ is universally bounded on the set $[z,y_3]\cap F_1''$.
Hence, we proved that $c$ is universally bounded on the set
$$
\Sigma_3'':=\bigcup_{z \in (\Sigma_2\cap F_1'')\cap \left(2(\Sigma_2'\cap F_1'')-y_3\right)}[z,y_3]\cap F_1''.
$$
Notice that, thanks to (b) above,
$\H^2(\Sigma_3 \setminus \Sigma_3'') \leq C\,\delta(A)^{\alpha_{n-1}/3}$.

Continue to iterate this construction by picking a point $z \in \Sigma_3''$
such that $\frac12(z+y_4) \in \Sigma_3'$, applying again 
Lemma \ref{lemma:1d}(ii) to the segments
$[z,y_4]$, and so on.
After $n-1$ steps we finally obtain a set $\Sigma_{n}''\subset F_1''$ such that
$\H^{n-1}(\Sigma_n\setminus \Sigma_{n}'')\leq C\,\delta(A)^{\alpha_{n-1}/3}$ and $c$ is universally bounded on $\Sigma_{n}''$.
Thanks to (a), this implies in particular that $\H^{n-1}(\Sigma_{n}'')\geq c_n/2$ 
provided $\delta(A)$ is sufficiently small.

\subsubsection*{Step 3-e: $A^*$ is bounded}
Since $A^*_y$ is a segment of length at most $2$ for any $y \in F_1''$,
we only need to prove that $c(y)$ is universally bounded for any $y \in F_1''$.

Fix $\bar y \in F_1''$. Since $F_1''$ is almost of full measure inside its convex hull (see \eqref{eq:coF_1''}),
$\co(F_1'')$ is universally bounded (see Step 3-a),
and $\Sigma_{n}''$ is a simplex inside $\co(F_1'')$ of non-trivial measure, by a simple Fubini
argument we can find a point $\bar y' \in F_1''\cap (2F_1'' - \bar y)$ such that 
most of the points on the segment $[\bar y,\bar y']$ belong to $F_1''$,
and in addition $\H^1\bigl([\bar y,\bar y']\cap \Sigma_{n}'' \bigr) \geq c_n'$
for some dimensional constant $c_n'>0$.

By applying Lemma \ref{lemma:1d}(i) to the function $[0,1]\ni \tau \mapsto \frac{1}{3}c(\tau \bar y+(1-\tau)\bar y')$, we deduce that $|c-\ell| \leq 3M$ on $[\bar y,\bar y']\cap  F_1''$ for some
linear function $\ell$.
However, we already know that $c$ is universally bounded on $[\bar y,\bar y']\cap \Sigma_{n}''$,
so $\ell$ is universally bounded there. Since
this set has non-trivial measure, this implies that $\ell$ has to be universally bounded
on the whole segment $[\bar y,\bar y']$ (since $\ell$ is a linear function).
Hence $c$ is universally bounded on $[\bar y,\bar y']\cap  F_1''$ as well,
and this provides a universal bound for $c(\bar y)$, concluding the proof.\\

\begin{remark}
\label{rmk:A bdd}
From the boundedness of $A^*$ we can easily prove that,
after the affine transformation described above,
 $A$ is bounded as well.  
More precisely, let $R>0$ be such that $A^*\subset B_R$.
We claim that (for $\d(A)$ sufficiently small)
\begin{equation}\label{eq:A bounded}
A \subset B_{3R}.
\end{equation}
Indeed, we deduce from $|A\setminus A^*| \le C\d(A)^{1/2}$ (see \eqref{eq:AA*}) 
that  $|A \setminus B_R|  \le C\d(A)^{1/2}$.    Hence, if 
 there is a point $x\in A \setminus B_{3R}$, since
 \[
\frac{A+A}{2} \supset  (A\cap B_R) \cup \biggl(\frac{(A\cap B_R)+x}{2}\biggr),
\]
and the two sets in the right hand side are disjoint, we get
\[
\biggl|\frac{A + A}{2}\biggr| \geq\bigl(1+2^{-n}\bigr)\bigl(1-C\d(A)^{1/2}\bigr)|A|
\]
which implies $\d(A) \ge 2^{-n-1}$, a contradiction.
\end{remark}

\subsection*{Step 4: $A^*$ is close to a convex set}
We show that  there exists a convex set
$\K\subset \R^n$ such that
\begin{equation}
\label{eq:A*K}
|\K \Delta A^*|\leq C \,\delta(A)^{\frac{\alpha_{n-1}\sigma_n}{8(n-1)}}.
\end{equation}
Before beginning the proof, let us recall some of the main properties of $A^*$
that we proved so far, and which will be used in the argument below.

First of all, $A^*$ is a $t$-convex set of the form
$$
A^*=\bigcup_{y \in F_1''}\{y\}\times [a(y),b(y)],
$$
where $F_1''$ is compact (see \eqref{eq:def F1''}), it is close to its convex hull $\co(F_1'')$
(see \eqref{eq:coF_1''}),
and $F_1''+F_1''$ is even closer to $2F_1''$ (see \eqref{eq:2F_1''}).
In addition, by Step 3,
up to an affine transformation as in \eqref{eq:affine}
we can assume that $\co(F_1'')$ is comparable to a ball whose radius is
bounded from above and below by two dimensional constants  (see \eqref{eq:Ks ball}),
and that $A^*\subset B_{(n-1)r}\times [-M,M]$  for some $M>0$ universal.
Finally, $a(y)$ and $b(y)$ satisfy \eqref{eq:y'y''y}.

In order to simplify the notation, we denote $\Omega:=\co(F_1'')$
and $F:=F_1''$.
Hence, by what we just said, 
$$
A^*=\bigcup_{y \in F}\{y\}\times [a(y),b(y)],\qquad \text{$F$ compact,}
$$
\begin{equation}
\label{eq:coF}
\Omega=\co(F),\quad
 \H^{n-1}(\Omega\setminus F) \leq
C\,\delta(A)^{\a_{n-1}/2},
\end{equation}
\begin{equation}
\label{eq:FF}
\ \H^{n-1}\bigl((F+F)\setminus 2F\bigr) \leq C\,\delta(A)^{1/2},
\end{equation}
\begin{equation}
\label{eq:Ks ball 2}
B_r\subset \Omega \subset B_{(n-1)r},\qquad 1/C_n <r<C_n,
\end{equation}
\begin{equation}
\label{eq:abM}
-M\leq a(y) \leq b(y) \leq M\qquad \forall\,y \in F,
\end{equation}
and 
\begin{equation}
\label{eq:ab}
a\biggl(\frac{y'+y''}2\biggr) - 8 \delta(A)^{1/2}\leq  \frac{a({y'})+a({y''})}{2} \leq \frac{b({y'})+b({y''})}{2} \leq b\biggl(\frac{y'+y''}2\biggr) + 8 \delta(A)^{1/2}
\end{equation}
whenever $y',y'',\frac{y'+y''}{2} \in F$.

Our goal is to show that $b$ (resp. $a$) is $L^1$-close to a concave (resp. convex) function
defined on $\Omega.$ Being the argument completely symmetric, we focus just on $b$.

\subsubsection*{Step 4-a: Making $b$ uniformly concave at points that are well separated}
Let $\beta\in(0,1/6]$ to be fixed later, and define $\var :\Omega \to \R$ as 
\begin{equation}
\label{eq:var}
\var(y):=
\left\{
\begin{array}{ll}
b(y) +2M - 20\,\delta(A)^{\beta}|y|^2& y \in F,\\
0 &y \in \Omega\setminus F.
\end{array}
\right.
\end{equation}
Notice that, because of \eqref{eq:abM} and \eqref{eq:ab}, we have
$0 \leq \var \leq 3M$ and
$$
\frac{\var(y')+\var(y'')}{2} \leq  \var \biggl(\frac{y'+y''}2\biggr)
+8\,\delta(A)^{1/2}
- 5\,\delta(A)^{\beta}|y'-y''|^2 \qquad \forall\,\,y',y'',\frac{y'+y''}{2} \in F,
$$
which implies in particular that
\begin{equation}
\label{eq:barvar}
\frac{\var(y')+\var(y'')}{2} \leq  \var \biggl(\frac{y'+y''}2\biggr)
+8\,\delta(A)^{1/2} \qquad \forall\,\,y',y'',\frac{y'+y''}{2} \in F,
\end{equation}
and (since $\beta \leq 1/6$)
\begin{equation}
\label{eq:barvar2}
\frac{\var(y')+\var(y'')}{2} < \var \biggl(\frac{y'+y''}2\biggr)-\delta(A)^{\beta}|y'-y''|^2 \qquad \forall\,\,y',y'',\frac{y'+y''}{2} \in F,\,\,|y'-y''| \geq 2\delta(A)^{\beta},
\end{equation}
that is $\var$ is  uniformly concave on points of $F$ that are at least $2\delta(A)^{\beta}$-apart.

\subsubsection*{Step 4-b: Constructing a concave function that should be close to $\var$}
Let us take $\gamma \in (0,1/4]$ to be fixed later,
and define
$$
\bar \var (y):=\min\{\var(y),h\},
$$
where $h \in [0,3M]$ is given by 
\begin{equation}
\label{eq:def h}
h:=\inf\bigl\{t>0 : \H^{n-1}(\{\var >t\})
\leq \delta(A)^{\gamma}\bigr\}.
\end{equation}
Since $0\leq \var \leq 3M$, we get
\begin{equation}
\label{eq:var bar}
\int_{\Omega} [\var(y) - \bar \var(y)]\,dy =\int_h^{3M} \H^{n-1}(\{\var>s\})\,ds \leq  3M \delta(A)^{\gamma}. 
\end{equation}
Notice that $\bar \var$ still satisfies \eqref{eq:barvar},
and it also satisfies \eqref{eq:barvar2} whenever
$\var((y'+y'')/2) = \bar\var((y'+y'')/2) < h$.

Finally, we define $\Phi:\Omega\to [0,h]$ to be the concave envelope of 
$\bar\var$, that is, the infimum
among all linear functions that are above $\bar\var$ in $\Omega$.
Our goal is to show that $\Phi$ is $L^1$-close to $\bar\var$ (and hence to $\var$).

\subsubsection*{Step 4-c: The geometry of contact sets of supporting hyperplanes}
Let $y$ belong to the interior of $\Omega$, and let $L$
be the linear function representing the supporting hyperplane for $\Phi$ at $y$, that is,
$L\ge \Phi$ in $\Omega$, and  $L(y) = \Phi(y)$.

Let $X:=\{\Phi=L\}\cap \Omega$. Observe that $X$ is a convex compact set (since $\Omega$
is convex and compact, being the convex hull of the compact set $F$) and $y \in X$.
Since $\Phi$ is the concave envelope of $\bar\var$, by
Caratheodory's theorem \cite[Theorem 1.1.4]{Sch}
there are $m$ points $y_1, \ldots, y_m \in X$, with $m \leq n$, 
such that $y \in \co(\{y_1, \ldots, y_m\})$ 
and all $y_j$'s are contact points:
\[
\Phi(y_j) =L(y_j)= \bar \var(y_j),  \qquad j = 1, \ldots, m.
\]
Observe that, because of \eqref{eq:abM} and \eqref{eq:var}, $\var>0$ on $F$,
and $\var=0$ on $\Omega\setminus F$.  We show next that $y_j\in F$ for all $j$.  

Fix $j$.  
Since $\Omega=\co(F)$ and $F$ is compact, we can apply Caratheodory's
theorem again to find $\ell$ points $z_1,\ldots,z_{\ell} \in F$, 
with $\ell \leq n$, such that $y_j\in \co(\{z_1,\ldots,z_\ell\})$.
But $z_i\in F$ implies $\Phi(z_i)\geq \bar\var(z_i) >0$, and hence  by concavity $\Phi(y_j)>0$. 
It follows that $\bar \var(y_j)= \Phi(y_j)>0$, and therefore $y_j\in F$.

In summary, every point in the interior of $\Omega$ belongs to a simplex $S$ 
such that
$$
S:=\co (\{y_1,\ldots, y_m\}),\qquad y_j \in \{\Phi = L =\bar \var\} \cap F,\qquad m\leq n.
$$

\subsubsection*{Step 4-d: The set $\{\Phi=\bar\var\}$ is $K\d(A)^\beta$ 
dense in $\Omega \setminus \co(\{\bar\var > h-K\d(A)^\beta\})$}

Let $\beta \in (0,1/6]$ be as in Step 4-b. We claim that there exists a dimensional constant $K>0$
such that the following hold, provided $\beta$ is sufficiently small (the smallness depending only on the dimension): For any $y \in \Omega$:\\
- either there is $x \in \{\Phi=\bar\var\}\cap \Omega$ with $|y-x|\leq K\delta(A)^\beta$;\\
- or $y$ belongs to 
the convex hull of the set $\{\bar\var > h-K\delta(A)^\beta\}$.

To prove this, we define
$$
\Omega_\beta:=\bigl\{y \in \Omega:\dist\bigl(y,\partial \Omega\bigr) \geq \delta(A)^\beta\bigr\}.
$$
Of course, up to enlarge the value of $K$, it suffices to consider the case when $y\in \Omega_\beta$. 

So, let us fix $y \in \Omega_\beta$.
Since $\Omega$ is a convex set comparable to a ball
of unit size (see \eqref{eq:Ks ball 2}) and $\Phi$ is 
a nonnegative concave function bounded by $3M$ inside $\Omega$,
there exists a dimensional constant $C'$ such that, for every linear 
function $L\ge \Phi$ satisfying $L(y) = \Phi(y)$, we have
\begin{equation}
\label{eq:bound gradient f}
|\nabla L| \leq \frac{C'}{\delta(A)^\beta} .
\end{equation}
By Step 4-c, there are $m\leq n$ points $y_1, \ldots, y_{m} \in F$
such that $y \in S:=\co(\{y_1, \ldots, y_{m}\})$, and
all $y_j$'s are contact points:
\[
\Phi(y_j) =L(y_j)= \bar \var(y_j),  \qquad j = 1, \ldots, m.
\]
If the diameter of $S$ is less than $K\d(A)^\beta$, then its
vertices are contact points
within $K\d(A)^\beta$ of $y$ and we are done. 

Hence, let us assume that the diameter
of $S$ is at least $K\d(A)^\beta$.
We claim that 
\begin{equation}\label{eq:x_i}
\bar\var(y_i) > h-K\delta(A)^\b\qquad \forall\, i =1,\ldots,m.
\end{equation}
Observe that, if we can prove \eqref{eq:x_i}, then
$$
y\in S\subset 
\co(\{\bar\var > h-K\delta(A)^\beta\}),
$$
and we are done again.

It remains only to prove \eqref{eq:x_i}.   To begin the proof, given 
$i \in \{1,\ldots,m\}$, take
$j \in \{1,\ldots,m\}$
such that $|y_i-y_j| \ge K\d(A)^\beta/2$ (such a $j$ always exists because of the assumption on the diameter of $S$).  We rename $i=1$ and
$j=2$.  

Let $N \in \N$ to be chosen, and for $x\in \Omega$ define
$$
H_N(x) :=F\cap( {2F-x} )\cap (4F - 3x) \cap \ldots \cap (2^{N}F-(2^{N}-1)x) = \bigcap_{k=0}^N (2^kF - (2^k-1)x).
$$
Observe that, since $\Omega$ is convex,
\begin{align*}
\H^{n-1}\bigl(\Om \setminus (2^kF - (2^k-1)x)\bigr) &
= 2^{k(n-1)} \H^{n-1}\bigl((2^{-k}\Om + (1-2^{-k})x)\setminus F\bigr) \\
&\leq 
2^{k(n-1)} \H^{n-1}(\Om\setminus F),
\end{align*}
so, by \eqref{eq:coF},
\begin{equation}
\label{eq:HN}
\H^{n-1} \bigl(\Om \setminus H_N(x)\bigr) \le \sum_{k=0}^N2^{k(n-1)} \,\H^{n-1}(\Om\setminus F) \leq C\, 2^{N(n-1)}\d(A)^{\a_{n-1}/2}.
\end{equation}

Let $w_0 \in H_N(y_2)$, and define $w_k := \frac12(y_2+ w_{k-1})$, $k=1,\, 2, \, 
\ldots, \, N$. Since $w_0 \in H_N(y_2)$ we have
\begin{equation}
\label{eq:wk}
w_k = (1-2^{-k})y_2 + 2^{-k}w_0
\in F\qquad \forall\,k=0,\ldots,N.
\end{equation}
Then, since $y_2 \in F$ and by \eqref{eq:wk},
we can apply iteratively \eqref{eq:barvar} to get (recall that $0\leq \bar \var\leq 3M$)
\begin{equation}
\label{eq:varz1}
\begin{split}
\bar \var(w_N)&\geq \bar \var(y_2)/2+\bar\var(w_{N-1})/2-8\delta(A)^{1/2}\\
&\geq  (1-1/4)\bar \var(y_2)+\bar\var(w_{N-2})/4-\bigl(1+1/2\bigr)\,8\delta(A)^{1/2}\\
&\geq \ldots\\
&\geq (1-2^{-N})\bar\var(y_2)+2^{-N}\bar\var(w_0)- 16\delta(A)^{1/2}\\
&\geq (1-2^{-N})\bar\var(y_2) -16\delta(A)^{1/2}\\
&\geq \bar\var(y_2)-C\,\left(2^{-N}+\delta(A)^{1/2}\right).
\end{split}
\end{equation}
In addition, since  the diameter of $F$ is bounded (see \eqref{eq:coF} and
\eqref{eq:Ks ball 2}), $|w_N- y_2|\leq C\, 2^{-N}$. 

Let us choose $N$ such that
$2^N=  c' \delta(A)^{-\frac{\alpha_{n-1}}{2(n-1)}}$ for some small dimensional constant $c'>0$.
In this way, from \eqref{eq:Ks ball 2} and \eqref{eq:HN} we get
\[
\H^{n-1}\bigl(H_N(y_2)\bigr) \ge c_n/2,
\]
which implies 
$$
\H^{n-1}\bigl((1-2^{-N})y_2+2^{-N}H_N(y_2)\bigr) \geq 2^{-(n-1)N}c_n/2= \frac{c_n}{2(c')^{n-1}}\delta(A)^{\alpha_{n-1}/2}.
$$
Hence, since by convexity of $\Omega$ and \eqref{eq:coF}
$$
\H^{n-1}\bigl(\Omega \setminus \bigl(F\cap(2F-y_1)\bigr)\bigr) \leq 3
\H^{n-1}(\Omega \setminus F)\leq 3C\,\delta(A)^{\alpha_{n-1}/2},
$$
we see that the set 
$F\cap (2F-y_1)\cap \bigl((1-2^{-N})y_2+2^{-N}H_N(y_2)\bigr)$ is nonempty provided $c'$ is sufficiently small.
So, let $x_2$ be an arbitrary point inside this set.
Observe that, with this choice, 
$$
|x_2-y_2|\leq  C\,\delta(A)^{\frac{\alpha_{n-1}}{2(n-1)}},\qquad z_1:=\frac{y_1+x_2}{2} \in F.
$$
We now prove \eqref{eq:x_i}: since $L$ has gradient of order at most $C\,\d(A)^{-\beta}$ (see \eqref{eq:bound gradient f}), we have
\[
|L(z_{1}) - L((y_1+y_2)/2)| \le C\,\d(A)^{-\beta}|x_2-y_2| \leq C\,\delta(A)^{\frac{\alpha_{n-1}}{2(n-1)}-\beta}
\]
Hence, since $y_1$ and $y_2$ are contact points and $L \geq \bar\var$,
using \eqref{eq:varz1} we get
\begin{equation}
\label{eq:mid}
\begin{split}
\frac{\bar\var(y_1) + \bar\var(x_2) }{2}
& \ge \frac{\bar\var(y_1) + \bar\var(y_2) }{2}-C\,\delta(A)^{\frac{\alpha_{n-1}}{2(n-1)}}\\
&= L((y_1+y_2)/2)-C\,\delta(A)^{\frac{\alpha_{n-1}}{2(n-1)}}\\
&\geq L(z_1) -C\,\delta(A)^{\frac{\alpha_{n-1}}{2(n-1)}-\beta}\\
&\geq \bar\var(z_1) -C\,\delta(A)^{\frac{\alpha_{n-1}}{2(n-1)}-\beta}.
\end{split}
\end{equation}
We now claim that, for some suitable choice of $K>0$ and $\beta \in \left(0,\frac{\alpha_{n-1}}{4(n-1)}\right]$, we can infer that
$\bar\var(z_1)=h$.
Observe that, if we can do so, then since $\bar\var \leq h$ is follows immediately
from \eqref{eq:mid} that both $\bar\var(y_1)$ and $\bar\var(x_2)$ have to be greater
than $h-C\,\delta(A)^{\frac{\alpha_{n-1}}{2(n-1)}-\beta} \geq h-C\,\delta(A)^\beta$, proving \eqref{eq:x_i}.

So, let us show that  $\bar\var(z_1)=h$. If not, we could apply
\eqref{eq:barvar2} with $y'=y_1$, $y''=x_2$, and $\var=\bar\var$, to get
$$
\bar\var(z_{1}) 
\ge \frac12(\bar\var(y_1) + \bar\var(x_2)) + \d(A)^\beta|y_1-x_2|^2.
$$
Since $|y_1-x_2| \geq |y_1-y_2|/2 \geq K\delta(A)^\beta/4$, this implies that
$$
\bar\var(z_{1}) 
\ge \frac12(\bar\var(y_1) + \bar\var(x_2)) + \frac{K^2}{16}\d(A)^{3\beta},
$$
which contradicts \eqref{eq:mid} provided we choose
$\beta= \frac{\alpha_{n-1}}{8(n-1)}$ and $K$ sufficiently large.

This concludes the proof with the choice
\begin{equation}
\label{eq:beta}
\beta:=\frac{\alpha_{n-1}}{8(n-1)}.
\end{equation}

\subsubsection*{Step 4-e: Most of the level sets of $\bar\var$ are close to their convex hull}

Let us now define the nonnegative function $\psi:\Omega\to \R$ as
$$
\psi(y):=\sup_{y'+y''=2y,\,\,y',y'' \in \Omega} \min\{\bar\var(y'),\bar\var(y'')\}.
$$
Notice that:\\
- $0 \leq \bar \var \leq \psi$ (just pick $y'=y''=y$);\\
- $\psi\leq \bar\var + 8\d(A)^{1/2}$ on $F$,
and $\psi=0$ outside $(F+F)/2$
(since by \eqref{eq:var} we have $\min\{\bar\var(y'),\bar\var(y'')\}=0$
unless both $y'$ and $y''$ belong to $F$, and then use \eqref{eq:barvar});\\
- $\psi \leq h$ (since $\bar \var \leq h$).\\
Thus, thanks to \eqref{eq:FF} we get
\begin{equation}
\label{eq:var psi}
\begin{split}
\int_{\Omega} \bar\var(y)\,dy &\leq \int_{\Omega} \psi(y)\,dy =\int_{(F+F)/2}\psi(y)\,dy\\
&\leq \int_{F} \bar\var(y)\,dy +
8\delta(A)^{1/2}\H^{n-1}(F)+h\,\H^{n-1}\bigl((F+F)\setminus 2F\bigr)\\
&\leq \int_{\Omega} \bar\var(y)\,dy+ C \,\d(A)^{1/2}.
\end{split}
\end{equation}
Since $\{\psi>s\}\supset \frac{\{\bar\var>s\}+\{\bar\var>s\}}{2}$,
we can apply Theorem \ref{thm:main} with $n-1$ to the level sets of $\bar\var$:
if we define
$$
H_1:=\bigl\{s :\H^{n-1}(\{\psi>s\})-\H^{n-1}(\{\bar\var>s\})<\delta_{n-1}\H^{n-1}(\{\bar\var>s\})\bigr\},
\qquad H_2:=[0,h]\setminus H_1,
$$
by \eqref{eq:var psi} and Fubini's Theorem we get
$$
c_{n-1}^{1/\alpha_{n-1}}\int_{H_1} \frac{\H^{n-1}\bigl(\co(\{\bar\var >s\})\setminus \{\bar\var >s\}\bigr)^{1/\a_{n-1}}}{\H^{n-1}(\{\bar\var >s\})^{1/\alpha_{n-1}-1}}\,ds+ \delta_{n-1}\int_{H_2}\H^{n-1}(\{\bar\var >s\})\,ds \leq C \,\d(A)^{1/2},
$$

Recalling the definition of $h$ (see \eqref{eq:def h}), we have
\[
\d(A)^{\gamma} \le \H^{n-1}(\{\bar \var >s\}) \le \H^{n-1}(\Omega) \le C, \qquad
0 \le s < h.
\]
Thus
\begin{equation}
\label{eq:H2}
\H^1(H_2) \leq C\,\delta(A)^{1/2-\gamma}\leq C\,\delta(A)^{1/4}
\end{equation}
(recall that, by assumption, $\gamma \leq 1/4$),
and
$$
\int_{H_1} \H^{n-1}\bigl(\co(\{\bar\var >s\})\setminus \{\bar\var >s\}\bigr)^{1/\a_{n-1}}\,ds\leq
C\,\d(A)^{1/2},
$$
so by H\"older's inequality (notice that $1/\a_{n-1} \geq 1$)
\begin{equation}
\label{eq:H1}
\int_{H_1} \H^{n-1}\bigl(\co(\{\bar\var >s\})\setminus \{\bar\var >s\}\bigr)\,ds\leq  C\,\d(A)^{\a_{n-1}/2}.
\end{equation}

\subsubsection*{Step 4-f: $b$ is $L^1$-close to a concave function}
Notice that, since the sets $\{\bar\var>s\}$ are decreasing in $s$, so are their convex hulls
$\co(\{\bar\var >s\})$. Hence, we can
define a new function $\xi:\Omega \to \R$ with convex level sets given by 
$$
\{\xi>s\}:=\co(\{\bar\var >s\})\quad \text{if $s \in H_1$},
\qquad \{\xi>s\}:=\bigcap_{\tau \in H_1,\,\tau<s}\co(\{\bar\var >\tau \}) \quad \text{if $s \in H_2$},
$$
Then by \eqref{eq:H2} and \eqref{eq:H1} we see that
$\xi$ satisfies
\begin{equation}
\label{eq:xi}
0\leq \bar\var \leq \xi,\quad \int_{\Omega} |\xi-\bar\var| \leq C\,\delta(A)^{\alpha_{n-1}/2}, \quad \xi \le \Phi \ (\text{the convex envelope of $\bar\var$}).
\end{equation}
Also, because of \eqref{eq:def h}, we see that
\begin{equation}
\label{eq:level xi}
\H^{n-1}(\{\xi>s\}) \geq \delta(A)^\gamma \qquad \forall\,0\leq s<h.
\end{equation}
Since by Step 4-d the contact set $\{\Phi = \bar\var\}$ is $K\delta(A)^{\b}$-dense
outside the set
$$
\co(\{\bar\var > h-K\delta(A)^\beta\})\subset \{\xi>h-K\delta(A)^\beta\},
$$
the same is true for the contact set $\{\Phi=\xi\}$.  

We claim that there exist dimensional constants $K',\eta>0$ such that, for any $0 \leq s <h-K\delta(A)^\beta$, each level set 
$\{\Phi>s\}$ is contained in a $K'\delta(A)^{\eta}$-neighborhood of $\{\xi>s\}$.

Indeed, if this was not the case, we could find a point $y \in \{\Phi>s\}$
such that ${\rm dist}(y,\{\xi>s\})>K'\delta(A)^{\eta}$. We now distinguish between the cases $n=2$ and $n \geq 3$.

If $n=2$ the sets $\{\xi>s\}$ and $\{\Phi>s\}$ are both intervals,
so we can find a point $z \in \{\Phi>s\}\setminus \{\xi>s\}$ such that $|y-z|\geq K'\delta(A)^{\eta}$
and the segment $[y,z]$
does not intersect $\{\xi>s\}$. Hence no contact points can be inside $[y,z]$,
which contradicts the density of $\{\Phi = \bar\var\}$ provided $\eta \leq \beta$ and $K'> K$.

If $n \geq 3$, since $\{\xi>s\}$ is a (universally) bounded convex set in $\R^{n-1}$,
by \eqref{eq:level xi} we deduce that it contains a $(n-1)$-dimensional ball $B_\rho(y')$
with $\rho:=c\delta(A)^{\gamma}$ for some dimensional constant $c>0$.
By convexity of $\{\Phi>s\}$, this implies that
$$
\mathcal C := \co\bigl(\{y\} \cup B_\rho(y')\bigr) \subset \{\Phi>s\}.
$$
Thanks to the fact that ${\rm dist}(y,\{\xi>s\})>K'\delta(A)^{\eta}$ and ${\rm diam}(\mathcal C)\leq C$,
we can find a $(n-1)$-dimensional ball $B_{r}(z)\subset \mathcal C$ such that
$B_{r}(z)\cap \{\xi>s\}=\emptyset$,
where $r:=c'K'\delta(A)^{\eta+\gamma}$ and $c'>0$ is a (small) dimensional constant.
Hence no contact points can be inside $B_{r}(z)$, 
and this contradicts the density of 
$\{\Phi = \bar\var\}$
provided $\eta+\gamma \leq \beta$ and $c'K'>K$.

In conclusion, the claim holds with the choices
\begin{equation}
\label{eq:gamma eta}
\eta:=\left\{
\begin{array}{ll}
\beta & \text{if $n=2$,}\\
\beta -\gamma & \text{if $n\geq 3$,}
\end{array}
\right.
\qquad 
K':=\frac{2K}{c'}.
\end{equation}
Since all level sets of $\xi$ are (universally) bounded convex sets, as a consequence of the claim we deduce that
$$
\H^{n-1}(\{\Phi>s\}) \leq \H^{n-1}(\{\xi>s\})+C\,\delta(A)^{\eta} \qquad \forall\,s \in [0,h-K\delta(A)^\beta].
$$
In addition, since $\xi \leq \Phi \leq h$, we obviously have
that $|\Phi-\xi|\leq K\delta(A)^\beta$ inside the set $\{\xi>h-K\delta(A)^\beta\}$.
Hence, by Fubini's Theorem,
$$
\int_{\Omega} |\Phi-\xi| \leq C\,\delta(A)^\eta
$$
(observe that, because of \eqref{eq:gamma eta}, $\eta \leq \beta$).
Since $\eta\leq \alpha_{n-1}/2$ (see \eqref{eq:beta} and \eqref{eq:gamma eta}), combining this estimate with \eqref{eq:xi} we get
$$
\int_{\Omega} |\Phi-\bar\var| \leq C\,\d(A)^\eta.
$$
In addition, since by construction $|\var(y)-2M - b(y)| \leq 20\delta(A)^\beta$ inside $F$ (see \eqref{eq:var}),
by \eqref{eq:var bar} we have
\begin{align*}
\int_{F}|\bar\var(y)-2M - b(y)|\,dy&
\leq \int_{F}|\bar\var(y)-\var(y)|\,dy
+\int_{F}|\var(y)-2M-b(y)|\,dy\\
&\leq C\,\left(\delta(A)^\gamma+\delta(A)^\beta\right).
\end{align*}
All in all, combining the two inequalities above, we see that
$$
\int_{F} |\Phi(y)-2M-b(y)|\,dy \leq C\,\bigl(\delta(A)^\gamma+\delta(A)^\eta\bigr).
$$
We now finally fix $\gamma$ and $\eta$: recalling \eqref{eq:gamma eta} and \eqref{eq:beta},
by choosing
\begin{equation}
\label{eq:sigman}
\gamma=\eta:=\sigma_n\beta, \qquad \text{with}\quad\sigma_n:=
\left\{
\begin{array}{ll}
1 & \text{if $n=2$,}\\
1/2 & \text{if $n\geq 3$,}
\end{array}
\right.
\end{equation}
we obtain
\begin{equation}
\label{eq:bPhi}
\int_{F} |\Phi(y)-2M-b(y)|\,dy \leq C\,\delta(A)^{\frac{\alpha_{n-1}\sigma_n}{8(n-1)}}.
\end{equation}
Applying the symmetric argument to $a(y)$,
we find a convex function $\Psi:\Omega \to [-3M,0]$
such that 
\begin{equation}
\label{eq:aPsi}
\int_{F} |\Psi(y)+2M-a(y)|\,dy\leq C\,\delta(A)^{\frac{\alpha_{n-1}\sigma_n}{8(n-1)}}.
\end{equation}

\subsubsection*{Step 4-g: Conclusion of the argument}
Let us define the convex set
$$
\K:=\bigl\{(y,t) \in \Omega\times \R: \Psi(y)+2M\leq t \leq \Phi(y)-2M\bigr\} \subset \Omega \times [-M,M].
$$
Then, using \eqref{eq:bPhi}, \eqref{eq:aPsi}, \eqref{eq:coF}, and
\eqref{eq:abM}, we get
\begin{align*}
|\K \Delta A^*|&\leq \int_{F} |\Phi(y)-2M-b(y)|\,dy
+\int_{F} |\Psi(y)+2M-a(y)|\,dy+ 2M\,\H^{n-1}(\Omega\setminus F)\\
&\leq C \,\delta(A)^{\frac{\alpha_{n-1}\sigma_n}{8(n-1)}}.
\end{align*}
This concludes Step 4.

\subsection*{Step 5: conclusion of the proof}

Combining \eqref{eq:AA*} and \eqref{eq:A*K}, we see that there exists a
convex set $\K$ such that $|\K\Delta A|\leq C \,\delta(A)^{\frac{\alpha_{n-1}\sigma_n}{8(n-1)}}$.

Observe that, by John's Lemma \cite{john}, after
replacing both $\K$ and $A$ by $L(\K)$ and $L(A)$,
where $L:\R^n\to \R^n$ is an affine transformation with $\det(L)=1$,
we can assume that $L(K)\subset B_R$ for some $R$ depending only on the dimension.

Also, after replacing $\K$ with $\K\cap \co(A)$ 
(which decreases the measure of the symmetric 
difference between $\K$ and $A$), we can assume that $\K\subset \co(A)$.

Following the argument used in the proof of \cite[Lemma 13.3]{christ1},
we now estimate $|\co(A)\setminus \K|$.
Indeed, let $x \in A\setminus \K$, denote by $x'\in \partial\K$ the closest point
in $\K$ to $x$, set $\rho:=|x-x'|=\dist(x,\K)$, and let $v \in \mathbb S^{n-1}$ be the unit normal to a supporting
hyperplane to $\K$ at $x'$, that is
$$
(z-x')\cdot v \leq 0 \qquad \forall\,z \in \K.
$$
Let us define $\K_\rho:=\{z \in \K:(z-x')\cdot v \geq -\rho\}$.
Observe that, since $\K$ is a bounded convex set with volume 
close to $1$, $|\K_\rho| \geq c \rho^n$ for some dimensional constant $c>0$.
Since $x \in A$ we have
$$
\frac{A+A}{2} \supset \frac{x+(\K_\rho\cap A)}{2} \cup (A\cap \K),
$$
and the two sets in the right hand side are disjoint.
This implies that
$$
\delta(A)+|A|=\biggl|\frac{A+A}{2} \biggr| \geq c2^{-n}\rho^n-|\K_\rho \setminus A| +|A\cap \K|
\geq c2^{-n}\rho^n +|A| - C \,\delta(A)^{\frac{\alpha_{n-1}\sigma_n}{8(n-1)}},
$$
from which we deduce that
$$
\rho\leq C\,\delta(A)^{\frac{\alpha_{n-1}\sigma_n}{8n(n-1)}}.
$$
Since $x$ is arbitrary, this implies that $A$ is contained inside the $\Bigl(C\,\delta(A)^{\frac{\alpha_{n-1}\sigma_n}{8n(n-1)}}\Bigr)$-neighborhood of $\K$.
By convexity, also $\co(A)$ has to be contained inside such a neighborhood, thus
$$
|\co(A)\setminus \K| \leq C\,\delta(A)^{\frac{\alpha_{n-1}\sigma_n}{8n(n-1)}}.
$$
Combining all the estimates together, we conclude that
$$
|\co(A)\setminus A|\leq C\,\delta(A)^{\alpha_n},\qquad \alpha_{n}:=\frac{\alpha_{n-1}\sigma_n}{8n(n-1)}.
$$
Recalling the definition of $\sigma_n$ in \eqref{eq:sigman},
since $\a_1=1$ (by Theorem \ref{thm:Freiman}) we deduce that $\alpha_n
=\frac{1}{8\cdot 16^{n-2}n!(n-1)!}$,
as desired.

\begin{remark} For use in the sequel \cite{fjBM}, we observe that if 
$\delta(A)\leq \delta_n$, 
then there exists a convex set $\K$ such that 
\begin{equation}
\delta(A)^{n\a_n} \geq c_n\frac{|\K\Delta A|}{|A|},\qquad \a_n:=\frac{1}{8\cdot 16^{n-2}n!(n-1)!}.
\end{equation}
In other words, if we only want to show that $A$ is close to some convex set 
(which may be different from $\co(A)$) the exponent in our stability estimate 
can be improved by a factor $n$.  This is a direct consequence of Steps 1-4, although 
Step 5 is still essential to close the induction argument
(for instance, in Step 4-f, the fact that the sets $\{\bar\var>s\}$ are close to their convex hulls 
is used in a crucial way). 
\end{remark}

\appendix

\section{Technical results}

\begin{lemma}
\label{lemma:observation}
Let $A\subset \R^n$ be a nonempty measurable set, write $\R^n=\R^k\times \R^{n-k}$,
denote by $\pi_k:\R^n\to \R^{k}$ the canonical projection, and for $y \in \R^k$ set
$A_y:=A\cap \pi_k^{-1}(y)$.
Then  
$$
\biggl(\sup_{y \in \R^{k}}\H^{n-k}(A_y)\biggr)\,\H^{k}(\pi_k(A)) \leq 2^n\bigl(1+\delta(A)\bigr)|A|.
$$
\end{lemma}
\begin{proof}
Let $\{y_j\}_{j \in \N}$ be a sequence of points such that $\H^{n-k}(A_{y_j}) \to \sup_{y \in \R^{k}}\H^{n-k}(A_y)$.
Then, since $A+A\supset \bigl(\{y_j\}\times A_{y_j}\bigr)+A$, we get
$$
2^n\bigl(1+\delta(A)\bigr) \geq |A+A| \geq \H^{n-k}(A_{y_j}) \,\H^{k}(\pi_k(A)),
$$
and the estimate follows letting $j \to \infty$.
\end{proof}

\begin{lemma}
\label{lemma:1d}
Let $E\subset \R$, and let $f:E\to \R$ be a bounded measurable function such that 
\begin{equation}
\label{eq:f}
\left| \frac{f(m')+f(m'')}{2} - f\left(\frac{m'+m''}{2}\right)\right| \leq 1 \qquad \forall\,m',m'',\frac{m'+m''}{2} \in E.
\end{equation}
Assume that there exist points $m_1,m_2 \in \R$ such that $m_1,m_2,\frac{m_1+m_2}{2} \in E$,
and  $|E\cap [m_1,m_2]| \geq (1-\e)|m_2-m_1|$.
Then the following hold provided $\e$ is sufficiently small (the smallness being universal):
\begin{enumerate}
\item[(i)] There 
exist a linear function $\ell:[m_1,m_2] \to \R$ and 
a universal constant $M$,
such that 
$$
|f-\ell |\leq M \qquad \text{in $E \cap [m_1,m_2]$}.
$$
\item[(ii)] If in addition $|f(m_1)|+|f(m_2)| \leq K$ for some constant $K$, then $|f|\leq K+M$
inside $E$.
\end{enumerate}
\end{lemma}
\begin{proof} Without loss of generality, we can assume that $[m_1,m_2]=[-1,1]$
and $E \subset [-1,1]$.  Given numbers $a\in\R$ and $b>0$, we write 
$a=O(b)$ if $|a|\leq Cb$ for some universal
constant $C$.\\

To prove (i), let us define
$$
\ell(m):=\frac{f(1)-f(-1)}{2} m + \frac{f(1)+f(-1)}{2},
$$
and set $F:=f-\ell$. 
Observe that $F(-1)=F(1)=0$, and $F$ still satisfies \eqref{eq:f}.
Hence, 
since by assumption $-1,0,1 \in E$, by \eqref{eq:f} we get $|F(0)| \leq 1$.
Let us extend $F$ to the whole interval $[-1,1]$ as $F(y)=0$ if $y \not\in E$, and set
$$
M:=\sup_{y \in E}|F(y)|.
$$
We want to show that $M$ is universally bounded.

Averaging \eqref{eq:f} (applied to $F$ in place of $f$) with respect to $m'' \in E$
and using that $|E\cap [-1,1]| \geq 2(1-\e)$, we easily obtain the following bound
$$
F(m')=-\frac{1}{2}\int_{-1}^1 F(m)\,dm 
+ 2\int_{(m'-1)/2}^{(m'+1)/2}F(m)\,dm+O(1)+O(\e M).
$$
Observe now that, since $|F(0)|\leq 1$, by \eqref{eq:f} applied with $f=F$, $m'=m$, 
and $m''=-m$,  we get
$$
|F(m)+F(-m)|\leq 4 \qquad \forall \,m \in E\cap -E.
$$
Since $|[-1,1]\setminus (E\cap -E)| \leq 4\e$ and $|F|\leq M$, we deduce that
$$
\int_{-1}^1 F(m)\,dm = \int_{E\cap -E}F(m)\,dm
+\int_{[-1,1] \setminus (E\cap -E)}F(m)\,dm = O(1)+O(\e M),
$$
hence
$$
F(m') = 2\int_{(m'-1)/2}^{(m'+1)/2}F(m)\,dm +O(1)+O(\e M)\qquad \forall\,m' \in E.
$$
As a consequence, if $m'<m''$,
\begin{equation}
\label{eq:c almost Lip}
\begin{split}
|F(m')-F(m'')|&\leq 2\int_{(m'-1)/2}^{(m''-1)/2}|F(m)|\,dm
+2 \int_{(m'+1)/2}^{(m''+1)/2}|F(m)|\,dm+O(1)+O(\e M)\\
&= O(M|m'-m''|)+O(1)+O(\e M)\qquad \forall\,m',m'' \in E.
\end{split}
\end{equation}

Now pick a point $m_0\in E$ such that
\begin{equation}
\label{eq:F large}
|F(m_0)|\geq M-1.
\end{equation}
Since $|F|$ is already bounded by $1$ at $-1,0,1$ we can assume that $m_0\neq -1,0,1$
(otherwise there is nothing to prove).
Without loss of generality we can suppose that $m_0\in (-1,0)$.
Then, since $|[-1,0]\setminus ((E-1)/2)| \leq \e$ and 
$|[-1,0] \setminus E|\leq 2\e$, 
we can find a point
$$
m_1 \in [-1,0] \cap ((E-1)/2)\cap E\cap [m_0-4\e,m_0+4\e],
$$
and by \eqref{eq:c almost Lip}
and \eqref{eq:F large} we get
$$
|F(m_1)|\geq M-O(1)-O(\e M).
$$
Hence, applying \eqref{eq:f} to $m'=-1$ and $m''=2m_1+1$ (observe that $m''\in E$
because $m_1 \in (E-1)/2$),
since
$$
F(m')=0,\qquad |F(m'')|\leq M, 
$$
we get
$$
M-O(1)-O(\e M)  \leq |F(m_1)| \le \left|F(m_1)- \frac{F(m') + F(m'')}{2} \right| + 
\frac M2 \le 1 + \frac{M}2
$$
which proves that $M$ is universally bounded
provided $\e$ is sufficiently small (the smallness being universal).
This proves (i).

To prove (ii), it suffices to observe that
if $|f(-1)|+|f(1)| \leq K$, then $|\ell| \leq K$ and we get
$$
|f| \leq |\ell|+|F| \leq K+M.
$$
\end{proof}

\section*{Appendix B. A qualitative version of Theorem \ref{thm:main}}

In \cite{christ1,christ2} M. Christ proved the
following result: Let $A,B \subset \R^n$ be two measurable sets 
such that $|A|$ and $|B|$ are
both uniformly bounded away from zero and infinity.
For any $\e>0$ there exists $\delta>0$ such that 
if
$$
|A+B|^{1/n}\leq |A|^{1/n}+|B|^{1/n} + \delta,
$$
then there exist a compact convex set $\K$, scalars $\a,\b>0$, and
vectors $a,b \in \R^n$, such that
$$
A\subset \a \K+a, \quad B\subset \b \K+b, \quad |(\a\K+a)\setminus A| \leq
\e,\quad |(\b\K+b)\setminus B| \leq \e.
$$
In the particular case when $A=B$ this result says that
$\frac{|\co(A)\setminus A|}{|A|}\to 0$
as $\delta(A) \to 0$.
Here, following the ideas in \cite{christ1}, 
we show that this last result follows very easily once one has proved that $A$
is bounded (which amounts to Steps 1-3 in our proof).
In particular,
if one is only interested in a qualitative statement,
the following simple argument allows one to skip Step 4.\\

Let $A\subset \R^n$ be a measurable set such that $|A|=1$ and $A\subset B_R$
for some fixed $R>0$.
We want to show that $|\co(A)\setminus A|\to 0$ as $\delta(A) \to 0$.

Given $k \in \{1,\ldots, n\}$, let $\pi_k:\R^n \to \R^{n-1}$ denote the
projection onto the hyperplane orthogonal to the $k$th axis:
$$
\pi_k(x_1,\ldots,x_n):=(x_1,\ldots,x_{k-1},x_{k+1},\ldots,x_n),
$$
and for $y \in \R^{n-1}$ define $A_y^k:=A\cap \pi_k^{-1}(y)$.

As in Step 1-a of Section \ref{sect:main}, we can write $\pi_k(A)$ as
$F_1^k\cup F_2^k$, where
$$
F_1^k:=\bigl\{y \in \pi_k(A):
\H^1(A_y^k+A_y^k)-2\H^1(A_y^k)<\H^1(A_y^k)\bigr\},\qquad
F_2^k:=\pi_k(A)\setminus F_1^k,
$$
and by Theorem \ref{thm:Freiman} applied to each set $A_y^k\subset\R$ we
deduce that
\begin{equation}
\label{eq:freiman Ayk}
\int_{F_1^k}\H^{1}\bigl(\co(A_y^k)\setminus A_y^k\bigr)\,dy+
\int_{F_2^k}\H^1(A_y^k)\,dy \leq 2^n\delta(A)
\end{equation}
(compare with \eqref{eq:freiman Ay}).
Set
$$
A_k^*:=\bigcup_{y \in F_1^k}  \co(A_y^k).
$$
Then by \eqref{eq:freiman Ayk} and Fubini Theorem we  get
\begin{equation}
\label{eq:A Ak}
|A\Delta A_k^*|\leq 2^n\delta(A).
\end{equation}
We  now follow the strategy in \cite[Lemma 12.1]{christ1} to show that
$A_k^*$
enjoys some fractional Sobolev regularity in the $k$th direction.

Observe that, since $A_k^*$ is a union of intervals, if we write
$f_k:=\chi_{A_k^*}$ and
$\co(A_y^k)=\{y_1,\ldots,y_{k-1}\}\times [a_y^k,b_y^k]\times \{y_k,\ldots,y_{n-1}\}$,
and we denote by $\mathscr F_k$ the Fourier
transform in the $k$-variable, that is
$$
\mathscr F_k[g](x_1,\ldots,x_{k-1},\xi_k,x_{k+1},\ldots,x_n):=
\int_\R e^{ix_k\xi_k}g(x)\,dx_k,
$$
then
$$
\mathscr
F_k[f_k](x_1,\ldots,x_{k-1},\xi_k,x_{k+1},\ldots,x_n)=\frac{e^{i\xi_kb_y^k}
- e^{i\xi_ka_y^k}}{\xi_k},\quad
y:=(x_1,\ldots,x_{k-1},x_{k+1},\ldots,x_n).
$$
Since $|b^k_y-a^k_y|\leq 2R$, we get that
$$
\bigl|\mathscr F_k[f_k]\bigr|  \leq \frac{C(R)}{1+|\xi_k|},
$$
so, using that the Fourier transform is an isometry in $L^2$
and $\H^{n-1}(\pi_k(A))\leq C(n,R)$, we obtain
\begin{align*}
\int_{\R^n} |\xi_k|^{1/2} \bigl| \hat f_k\bigr|^2(\xi)\,d\xi &=
\int_{\pi_k(A)} |\xi_k|^{1/2} \bigl| \mathscr F_k[f_k]\bigr|^2\,d x_1
\,\ldots \ ,\,dx_{k-1}\,d\xi_k\,d x_{k+1}\,\ldots\,dx_n \\
&\leq  \H^{n-1}(\pi_k(A)) \int_\R \frac{|\xi_k|^{1/2}}{(1+|\xi_k|)^2} \leq
C(n,R).
\end{align*}
Using again that the Fourier transform is an isometry in $L^2$, by
\eqref{eq:A Ak} we get
\begin{align*}
\int_{\R^n} \min\left\{|\xi_k|^{1/2}, \delta(A)^{-1}\right\}\bigl| \hat
\chi_A\bigr|^2(\xi)\,d\xi &\leq
2\,\delta(A)^{-1} \int_{\R^n} \bigl| \hat \chi_A - \hat
f_k\bigr|^2(\xi)\,d\xi
+2 \int_{\R^n} |\xi_k|^{1/2} \bigl| \hat f_k\bigr|^2(\xi)\,d\xi\\
&= 2 \, \delta(A)^{-1} \int_{\R^n}| \chi_A - \chi_{A_k^*}|^2(x)\,dx
+2 \int_{\R^n} |\xi_k|^{1/2} \bigl| \hat f_k\bigr|^2(\xi)\,d\xi\\
&\leq 2^{n+1}+ C(n,R).
\end{align*}
Since $k \in \{1,\ldots,n\}$ is arbitrary, this implies that
$$
\int_{\R^n} \min\left\{|\xi|^{1/2}, \delta(A)^{-1}\right\}\bigl| \hat
\chi_A\bigr|^2(\xi)\,d\xi \leq C(n,R).
$$
It is a standard fact in Sobolev spaces theory that, thanks to this estimate,
any sequence of sets $\{A_j\}_{j\in \N}$ with $\delta(A_j) \to 0$
is precompact in $L^2$ (see for instance the discussion
in \cite[Corollary 12.2]{christ1}).
Hence, up to a subsequence, $\chi_{A_j}$ converge in $L^2$
to some characteristic function $\chi_{A_\infty}$,
and it is not difficult to check that $\delta(A_\infty)=0$ (see for
instance \cite[Lemma 13.1]{christ1}).
By the characterization of the equality cases in the semi-sum inequality,
we deduce that $A_\infty$ is equal to its convex hull up to a set of
measure zero,
thus $|A_j \Delta \co(A_\infty)| \to 0$.
Arguing as in \cite[Lemma 13.3]{christ1}
or as in Step 5 of Section \ref{sect:main}, this actually implies that
$|\co(A_j)\setminus A_j| \to 0$,
proving the result.


\begin{thebibliography}{99}

\bibitem{cauchy}
Cauchy A. L.
Recherches sur les nombres.
\textit{J. \'Ecole Polytech.} 9 (1813), 99-116.

\bibitem{christ1}
Christ M.
Near equality in the two-dimensional Brunn-Minkowski inequality.
Preprint, 2012. Available online at {\tt http://arxiv.org/abs/1206.1965}

\bibitem{christ2}
Christ M.
Near equality in the Brunn-Minkowski inequality.
Preprint, 2012. Available online at {\tt http://arxiv.org/abs/1207.5062}

\bibitem{davenport}
Davenport H.
On the addition of residue classes.
\textit{J. London Math. Soc.} 10 (1935), 30-32.


\bibitem{fjBM}
Figalli, A.; Jerison, D. Quantitative stability of the Brunn-Minkowski inequality. Preprint, 2013.



\bibitem{fmpK}
Figalli, A.; Maggi, F.; Pratelli, A.
A mass transportation approach to quantitative isoperimetric inequalities.
\textit{Invent. Math.} 182 (2010), no. 1, 167-211.

\bibitem{fmpBM}
Figalli, A.; Maggi, F.; Pratelli, A.
A refined Brunn-Minkowski inequality for convex sets.
\textit{Ann. Inst. H. Poincar\'e Anal. Non Lin\'eaire} 26 (2009), no. 6, 2511-2519.

\bibitem{freiman}
Freiman, G. A.
The addition of finite sets. I. (Russian) 
\textit{Izv. Vyss. Ucebn. Zaved. Matematika}, 1959, no. 6 (13), 202-213. 

\bibitem{freiman2}
Freiman, G. A.
\textit{Foundations of a structural theory of set addition. }
Translated from the Russian. Translations of Mathematical Monographs, Vol 37. American Mathematical Society, Providence, R. I., 1973.


\bibitem{john}
John F.
{\em Extremum problems with inequalities as subsidiary conditions.}
In Studies and Essays Presented to R. Courant on his 60th
  Birthday, January 8, 1948, pages 187-204. Interscience, New York, 1948.

\bibitem{Sch}
Schneider, R. \textit{Convex bodies: the Brunn-Minkowski theory.} Encyclopedia of Mathematics and its Applications, 44. Cambridge University Press, Cambridge, 1993.


\bibitem{bookTao}
Tao, T.; Vu, V. \textit{Additive combinatorics.} Cambridge Studies in Advanced Mathematics, 105. Cambridge University Press, Cambridge, 2006.

\end{thebibliography}
\end{document}